\documentclass{amsart}

\usepackage[paperheight=11in, 
    paperwidth=8.5in,
    outer=1.35in,
    inner=1.35in,
    bottom=1.5in,
    top=1.5in]{geometry}

\usepackage{amsmath, amsxtra, amsthm, amsfonts, amssymb}
\usepackage{mathrsfs}
\usepackage{graphicx}
\usepackage{url}
\usepackage{color}
\usepackage{bbm}
\usepackage{tikz-cd}

\usepackage[T1]{fontenc}
\usepackage{enumitem}

 \usepackage[hidelinks,backref=page]{hyperref} 
 \hypersetup{
    colorlinks,
    citecolor=magenta,
    filecolor=magenta,
    linkcolor=blue,
    urlcolor=black
}

\usepackage{cleveref}

\theoremstyle{definition}
\newtheorem{thm}{Theorem}[section]

\newtheorem{lem}[thm]{Lemma}
\newtheorem{prop}[thm]{Proposition}
\newtheorem{defn}[thm]{Definition}

\newtheorem{eg}[thm]{Example}
\newtheorem{rem}[thm]{Remark}

\newtheorem{maintheorem}{Theorem}
\newtheorem{question}[thm]{Question}
\newtheorem{prop-defn}[thm]{Proposition-Definition}

\newcommand{\RR}{\mathbb R}

\newcommand{\PP}{\mathbb P}
\newcommand{\ZZ}{\mathbb Z}
\newcommand{\M}{\mathrm{M}}
\newcommand{\D}{\mathrm{D}}
\newcommand{\be}{\mathbf{e}}
\newcommand{\kk}{\mathbbm k}
\renewcommand{\L}{[L]}

\renewcommand{\emptyset}{\varnothing}

\newcommand{\one}{{\mathbf 1}}

\title{K-classes of delta-matroids and equivariant localization}
\author{Christopher Eur, Matt Larson, Hunter Spink}

\begin{document}

\maketitle

\vspace{-10pt}

\begin{abstract}
Delta-matroids are ``type B'' generalizations of matroids in the same way that maximal orthogonal Grassmannians are generalizations of Grassmannians.  A delta-matroid analogue of the Tutte polynomial of a matroid is the interlace polynomial.  We give a geometric interpretation for the interlace polynomial via the $K$-theory of maximal orthogonal Grassmannians.  To do so, we develop a new Hirzebruch--Riemann--Roch-type formula for the type B permutohedral variety.
\end{abstract}

\section{Introduction}

For a nonnegative integer $n$, let $[n] = \{1, \dots, n\}$, and for a subset $S\subseteq [n]$, let $\be_S = \sum_{i\in S} \be_i$ be the sum of the corresponding standard basis vectors in $\RR^n$.
Let $[\bar n] = \{\bar 1, \ldots, \bar n\}$, and 
consider $[n, \bar n] = [n]\sqcup [\bar n]$ equipped with the involution $i \mapsto \bar i$.
Writing $\be_{\bar i} = -\be_i$, let $\be_S = \sum_{i\in S} \be_i$ for a subset $S\subseteq [n,\bar n]$.
A subset $S\subseteq [n,\bar n]$ is \emph{admissible} if $\{i,\bar i\} \not\subset S$ for all $i\in [n]$.  Note that a \emph{maximal admissible subset} of $[n,\bar n]$ has cardinality $n$.

\begin{defn}
A \emph{delta-matroid} $\D$ on $[n,\bar n]$ is a nonempty collection $\mathcal F$ of maximal admissible subsets of $[n,\bar n]$ such that
each edge of the polytope
\[
P(\D) = \text{the convex hull of }\{\be_{B\cap [n]} : B\in \mathcal F\} \subset \RR^n
\]
is a parallel translate of $\be_i$ or $\be_i\pm \be_j$ for some $i,j\in[n]$.
\end{defn}

The collection $\mathcal F$ is called the \emph{feasible sets} of $\D$, and $P(\D)$ is called the \emph{base polytope} of $\D$.
One often works with the following translation of the twice-dilated base polytope
\[
\widehat{P(\D)} = 2P(\D) - (1, \dots, 1) = \text{the convex hull of $\{\be_B : B \in \mathcal F\}$} \subset\RR^n.
\]
Delta-matroids generalize matroids as the ``minuscule type B matroids'' in the theory of Coxeter matroids \cite{Gelfand1987a,BGW03}, and as ``2-matroids'' in the theory of multimatroids \cite{BouShelter}.
The Tutte polynomial of a matroid \cite{Tut67, Cra69} admits a delta-matroid analogue called the \emph{interlace polynomial}, introduced in \cite{InterlaceBollobas,Bridgeboom}.

\begin{defn}
For a delta-matroid $\D$ on $[n,\bar n]$ with feasible sets $\mathcal{F}$  and a subset $S \subseteq [n]$, let 
\[
d_{\D}(S) = \min_{B \in \mathcal{F}} \big(| S \cup (B\cap [n])| - |S\cap B\cap [n]| \big), \text{ the lattice distance between $\be_{S}$ and $P(\D)$}.
\]
Then, the \emph{interlace polynomial} $\operatorname{Int}_{\D}(v)\in \ZZ[v]$ of $\D$ is defined as
\[
\operatorname{Int}_{\D}(v) = \sum_{S \subseteq [n]} v^{d_\D(S)}.
\]
\end{defn}

Similar to the Tutte polynomial of a matroid, the interlace polynomial has several alternative definitions: it satisfies a deletion-contraction recursion \cite[Theorem 30]{Bridgeboom}, it is an evaluation of the rank generating function of a delta-matroid \cite{DeltaRank}, and $\operatorname{Int}_{\D}(v-1)$ has an activities description \cite{MorseActivity}. Additionally, its evaluation at $q=0$ gives the number of feasible sets. 
Here, we show that Fink and Speyer's geometric interpretation of Tutte polynomials via the $K$-theory of Grassmannians \cite{FinkSpeyer} also generalizes to interlace polynomials.
Let us first recall their result.

\medskip
Each $r$-dimensional linear space $L \subseteq \kk^n$ over a field $\kk$ gives rise to a matroid $\M$ on $[n]$ and a point $\L$ in the Grassmannian $\operatorname{Gr}(r; n)$.
The torus $T = (\kk^*)^n$ acts on $\operatorname{Gr}(r; n)$, and we consider the torus-orbit-closure $\overline{T\cdot \L}$ of $L$.
The $K$-class of the structure sheaf $[\mathcal{O}_{\overline{T \cdot \L}}]$ in Grothendieck ring $K(\operatorname{Gr}(r; n))$ of vector bundles on $\operatorname{Gr}(r;n)$ depends only on $\M$, and it admits a combinatorial formula which makes sense for any matroid $\M$ of rank $r$ on $[n]$.
This formula is used to define a class $y(\M) \in K(\operatorname{Gr}(r; n))$ such that $y(\M) = [\mathcal O_{\overline{T\cdot \L}}]$ whenever $\M$ has a realization $L$.
 
Now, consider the diagram

\begin{center}
\begin{tikzcd}
& \operatorname{Fl}(1, r, n-1; n) \arrow[dr] \arrow[ddr, "{\pi}_{1n}"'] \arrow[dl, "{\pi}_r"]& \\ 
 \operatorname{Gr}(r; n)& &  \operatorname{Fl}(1, n-1; n) \arrow[d, hook] \\ 
 & &\mathbb{P}^{n-1} \times \mathbb{P}^{n-1}
\end{tikzcd}
\end{center}
where $\pi_r$ and $\pi_{1n}$ are the natural forgetful maps.
Then \cite[Theorem 5.1]{FinkSpeyer} states that
$$\pi_{1n*} \pi_r^* \big( y(\M) \cdot [\mathcal{O}(1)] \big) = \operatorname{T}_\M(\alpha, \beta),$$
where $\mathcal{O}(1)$ is the line bundle on $\operatorname{Gr}(r; n)$ defining the Pl\"{u}cker embedding, $\alpha$ and $\beta$ are the $K$-classes of the structure sheaves of hyperplanes in each of the $\PP^{n-1}$ factors, and $\operatorname{T}_\M$ is the Tutte polynomial of $\M$. 
This result was subsequently generalized to Tutte polynomials of morphisms of matroids in \cite{CDMS18,FlagTutte}.
Here, we establish a similar geometric interpretation for the interlace polynomials of delta-matroids via the $K$-theory of maximal orthogonal Grassmanians.

\medskip
Let $\kk^{2n+1}$ have coordinates labelled $\bar{n}, \dotsc, \bar{1}, 0, 1, \dotsc, n$.
Let $q$ be the nondegenerate quadratic form on $\kk^{2n+1}$ given by $q(x) = x_1 x_{\bar{1}} + \dotsb + x_n x_{\bar{n}} + x_0^2$.
For $0\leq r \leq n$, 
let $\operatorname{OGr}(r;2n+1)$ be the \emph{orthogonal Grassmannian}, which is the subvariety of $\operatorname{Gr}(r;2n+1)$ consisting of isotropic $r$-dimensional subspaces, i.e.,
\[
\operatorname{OGr}(r;2n+1) = \{\text{$r$-dimensional linear subspaces $L\subset \kk^{2n+1}$ such that $q|_L$ is identically zero}\}.
\]
The action of the torus $T = (\kk^*)^n$ on $\kk^{2n+1}$ given by
\[
(t_1, \ldots, t_n) \cdot (x_{\bar n}, \ldots, x_{\bar 1}, x_0, x_1, \ldots, x_n) = (t_n^{-1}x_{\bar n}, \ldots, t_1^{-1}x_{\bar 1}, x_0, t_1x_1, \ldots, t_nx_n)
\]
preserves the quadratic form $q$, and hence induces a $T$-action on $\operatorname{OGr}(r;2n+1)$.
One has the $T$-equivariant Pl\"ucker embedding $\operatorname{OGr}(r;2n+1) \hookrightarrow \operatorname{Gr}(r;2n+1) \hookrightarrow \PP(\bigwedge^r \kk^{2n+1})$.

\medskip
The \emph{maximal orthogonal Grassmannian} is $\operatorname{OGr}(n;2n+1)$.  Points on $\operatorname{OGr}(n;2n+1)$ realize delta-matroids in the same way that points on the usual Grassmannian realize matroids.  More precisely, \cite[Proposition 6.2]{EFLS} \cite{Gelfand1987a} showed that the torus-orbit-closure $\overline{T\cdot \L}$ of a point $\L\in \operatorname{OGr}(n;2n+1)$, considered as a $T$-invariant subvariety of $\PP(\bigwedge^n \kk^{2n+1})$ via the Pl\"ucker embedding, has moment polytope $\mu(\overline{T\cdot \L})$ equal to $\widehat{P(\D)}$, where $\D$ is a delta-matroid with the set of feasible sets
\[
\{\text{maximal admissible $B\subset [n,\bar n]$ such that the $B$-th Pl\"ucker coordinate of $L$ is nonzero}\}.
\]
Using this polyhedral property, we construct for any (not necessarily realizable) delta-matroid $\D$ an element $y(\D)$ in the Grothendieck ring $ K(\operatorname{OGr}(n;2n+1))$ of vector bundles on $\operatorname{OGr}(n;2n+1)$ (see Proposition~\ref{prop:yDdef}).\footnote{We caution that, unlike the matroid case in \cite{FinkSpeyer}, the class $y(\D)$ of a delta-matroid $\D$ with a realization $[L]\in \operatorname{OGr}(n;2n+1)$ may not be equal to the $K$-class of the structure sheaf $[\mathcal O_{\overline {T\cdot \L}}]$, although it is closely related, see \Cref{prop:isVal} and \Cref{prop:OTL}.  For a detailed discussion of $[\mathcal O_{\overline{T\cdot \L}}]$, see \Cref{rem:OTL} and \Cref{sec:egs} .}

\smallskip
To relate the $K$-class $y(\D)$ to the the interlace polynomial, we consider the orthogonal partial flag variety $\operatorname{OFl}(1,n; 2n+1) \subset \operatorname{OGr}(1; 2n+1) \times \operatorname{OGr}(n; 2n+1)$.
Note that $\operatorname{OGr}(1; 2n+1)$ is a smooth quadric inside of $\operatorname{Gr}(1; 2n+1) = \mathbb{P}^{2n}$.  We have the diagram
\begin{center}
\begin{tikzcd}
& \operatorname{OFl}(1, n; 2n+1) \arrow[dr] \arrow[ddr, "{\pi}_1"'] \arrow[dl, "{\pi}_n"]& \\ 
 \operatorname{OGr}(n; 2n+1)& &  \operatorname{OGr}(1; 2n+1) \arrow[d, hook] \\ 
 & &\mathbb{P}^{2n}.
\end{tikzcd}
\end{center}
Let $\mathcal{O}(1)$ denote the ample line bundle that generates the Picard group of $\operatorname{OGr}(n;2n+1)$.
Its square $\mathcal O(2)$ defines the Pl\"ucker embedding $\operatorname{OGr}(n;2n+1) \hookrightarrow \operatorname{Gr}(n;2n+1) \hookrightarrow \PP(\bigwedge^n \kk^{2n+1})$.
This fact about $\mathcal O(2)$ follows from the description of the Picard group of $\operatorname{OGr}(n; 2n+1)$ in terms of the representation theory of $\operatorname{SO}(2n+1)$; see \cite[Section 2.8]{MR1782635} for a summary of general theory, and \cite[Chapter 19.4]{FultonHarris} for features particular to $\operatorname{SO}(2n+1)$.
The line bundle $\mathcal{O}(1)$ defines the Spinor embedding of $\operatorname{OGr}(n; 2n+1)$ into $\mathbb{P}^{2^n - 1}$. 
Recall that $K(\mathbb{P}^{2n})\simeq \mathbb{Z}[u]/(u^{2n+1})$, where $u$ is the structure sheaf of a hyperplane in $\PP^{2n}$. So we may represent any class in $K(\mathbb{P}^{2n})$ uniquely as a polynomial in $u$ of degree at most $2n$. 

\begin{maintheorem}\label{thm:FS}
Let $\operatorname{Int}_\D(v)\in \ZZ[v]$ be the interlace polynomial of a delta-matroid $\D$.  We have
$$\pi_{1*}\pi_n^* \big(y(\D) \cdot [\mathcal{O}(1)] \big) = u \cdot \operatorname{Int}_{\D}(u - 1) \in K(\mathbb{P}^{2n}).$$
\end{maintheorem}

To prove the theorem, in Proposition~\ref{prop:transfer} we transport the pullback-pushforward ${\pi_1}_*\pi_n^*(-)$ computation to a sheaf Euler characteristic $\chi(-)$ computation on a smooth projective toric variety $X_{B_n}$ known as the \emph{type B permutohedral variety} (Definition~\ref{defn:Bn}).
Then, to carry out the sheaf Euler characteristic computation, we establish the following new Hirzebruch--Riemann--Roch-type formula for $X_{B_n}$.  Let $A^\bullet(X_{B_n})$ be the Chow ring of $X_{B_n}$, with the degree map $\int_{X_{B_n}} \colon A^n(X_{B_n}) \overset\sim\to \ZZ$.

\begin{maintheorem}\label{thm:HRR}
There is an injective ring homomorphism $\psi \colon K(X_{B_n}) \to A^{\bullet}(X_{B_n})$, which becomes an isomorphism after tensoring with $\mathbb{Z}[\frac{1}{2}]$. For any $[\mathcal E]\in K(X_{B_n})$, the map $\psi$ satisfies
\[
\chi(X_{B_n}, [\mathcal{E}]) = \frac{1}{2^n}\int_{X_{B_n}}  \psi([\mathcal{E}]) \cdot (1+\gamma + \gamma^2 + \cdots + \gamma^n)
\]
where $\gamma$ is the anti-canonical divisor of $X_{B_n}$.
\end{maintheorem}

The map $\psi$ in Theorem~\ref{thm:HRR} is unrelated to the usual Chern character. It also differs from the Hirzebruch--Riemann--Roch-type isomorphism of \cite[Theorem C]{EFLS}, which is not as suitable for proving Theorem~\ref{thm:FS}.

\begin{question}
The $g$-polynomial \cite{Speyerg} of a matroid is an invariant of matroids that can be (conjecturally) used to give strong bounds on the number of pieces in a matroid polytope subdivision. The coefficients of the $g$-polynomial are certain linear combinations of the coefficients that are used to express $y(\M)$ in terms of structure sheaves of Schubert varieties in $K(\operatorname{Gr}(r; n))$. In \cite[Theorem 6.1]{FinkSpeyer}, the authors express the $g$-polynomial in terms of a computation similar to the one in Theorem~\ref{thm:FS}. Is there an invariant of delta-matroids which gives strong bounds on the number of pieces in a delta-matroid polytope subdivision?
\end{question}

The paper is organized as follows.  In Section~\ref{sec:Kclass}, we discuss equivariant $K$-theory and define $y(\D)$. We also discuss a key tool, the theory of \emph{valuative} invariants of delta-matroids, which we repeatedly use to reduce statements to the case of realizable delta-matroids. 
In Section~\ref{sec:3}, we prove Theorem~\ref{thm:HRR} and discuss certain class in $K(X_{B_n})$ which will be used in the proof of Theorem~\ref{thm:FS}. In Section~\ref{sec:4}, we prove Theorem~\ref{thm:FS}. In Section~\ref{sec:egs}, we give some examples and questions. 

\subsection*{Acknowledgements}
We thank Alex Fink, Steven Noble, Kris Shaw, and David Speyer for helpful conversations. We thank the referee for their helpful comments. 
The first author is partially supported by the US National Science Foundation (DMS-2001854 and DMS-2246518).  The second author is supported by an NDSEG graduate fellowship.

\section{$K$-classes of delta-matroids}\label{sec:Kclass}

Throughout, we will use localization for the torus-equivariant $K$-theory of toric varieties and flag varieties, for which one can consult \cite[\S2.2]{FinkSpeyer}, \cite[\S2.2]{FlagTutte}, or \cite[\S8]{CDMS18} along with the references therein.  Let $T = (\kk^*)^n$ for $\kk$ an algebraically closed field, and denote by $K_T(X)$ the $T$-equivariant $K$-ring of vector bundles on a $T$-variety $X$.  Identifying the character lattice of $T$ with $\ZZ^n$, we write $K_T(\operatorname{pt}) = \ZZ[T_1^{\pm 1}, \dots, T_n^{\pm 1}]$ for the equivariant $K$-ring of a point $\operatorname{pt}$.  For $\mathbf m = (m_1, \ldots, m_n) \in \ZZ^n$, we write $T^{\mathbf m} = T_1^{m_1}\cdots T_n^{m_n}$.

For a countable-dimensional $T$-representation $V\simeq \bigoplus_i \kk \cdot v_i$, where $T$ acts on $v_i$ by $t\cdot v_i = t^{\mathbf{m}_i}v_i$, the \emph{Hilbert series} $\operatorname{Hilb}(V) = \sum_i T^{\mathbf{m}_i}$ is the sum of the characters of the action, which is often a rational function.  For an affine semigroup $S\subseteq \ZZ^n$, we write $\operatorname{Hilb}(S) =  \operatorname{Hilb}(\kk[S]) = \sum_{\mathbf m \in S} T^{-\mathbf m}$.  Note the minus sign, which arise because for $\chi^{\mathbf m} \in \kk[S]$, we have $t\cdot \chi^{\mathbf m} = t^{-\mathbf m}\chi^{\mathbf m}$.

\subsection{$K$-classes on the maximal orthogonal Grassmannian}

We begin by recalling some facts about the $T$-action on $\operatorname{OGr}(n;2n+1)$, whose verification is routine and is omitted. Recall that we have set $\be_{\bar{i}} = - \be_i$. 
\begin{itemize}
\item The $T$-fixed points $\operatorname{OGr}(n;2n+1)^T$ of $\operatorname{OGr}(n;2n+1)$ are in bijection with maximal admissible subsets, where such a subset $B \subset [n,\bar n]$ corresponds to the isotropic subspace
\[
L_B = \{x\in \kk^{2n+1} : x_0 = 0 \text{ and } x_j = 0 \text{ for all } j \in [n,\bar n]\setminus B\}.
\]
Polyhedrally, by identifying $B\subset [n,\bar n]$ with $\be_{B\cap [n]}\in \RR^n$, we may further identify the $T$-fixed points with the vertices of the unit cube $[0,1]^n \subset \RR^n$.
\item Each $T$-fixed point $L_B$ admits a $T$-invariant affine chart $U_B \simeq \mathbb{A}^{n(n+1)/2}$, on which $T$ acts with characters in the finite set
\[
\mathcal T_B = \{-\be_i : i\in B\} \cup \{-\be_i-\be_j : i\neq j \in B\}.
\]
In particular, for $\mathbf v \in \mathcal T_B$ with $B' \subset [n,\bar n]$ such that $\be_{B'} = \be_B + 2\mathbf v$, we have an 1-dimensional $T$-orbit in $\operatorname{OGr}(n;2n+1)$ whose boundary points are $L_B$ and $L_{B'}$.
All 1-dimensional $T$-orbits of $\operatorname{OGr}(n;2n+1))$ arise in this way.
\end{itemize}
Now, the localization theorem applied to $K_T(\operatorname{OGr}(n;2n+1))$ states the following:

\begin{thm}\label{thm:localizationK} \cite[Corollary 5.11]{VezzosiVistoli}
The restriction map
\[
K_T(\operatorname{OGr}(n;2n+1)) \to K_T(\operatorname{OGr}(n;2n+1)^T) = \prod_{L_B \in \operatorname{OGr}(n;2n+1)^T} \ZZ[T_1^{\pm 1}, \dots T_n^{\pm 1}]
\]
is injective, and its image is
\[
\left\{ (f_B)_B \in \prod_{L_B \in \operatorname{OGr}(n;2n+1)^T} \ZZ[T_1^{\pm 1}, \dots T_n^{\pm 1}] : \begin{matrix}\text{for $\mathbf v \in \mathcal T_B$ with $B' \subset [n,\bar n]$ such that $\be_{B'} = \be_B + 2\mathbf v$}\\ f_B - f_{B'} \equiv 0 \text{ mod } (1 - T^{\mathbf v}) \end{matrix}\right\}.
\]
\end{thm}

For an equivariant $K$-class $[\mathcal E]\in K_T(\operatorname{OGr}(n;2n+1))$ and a maximal admissible subset $B$, we write $[\mathcal E]_B \in \ZZ[T_1^{\pm 1}, \ldots, T_n^{\pm 1}]$ for the $B$-th factor of the image of $[\mathcal E]$ under the restriction map in \Cref{thm:localizationK}.

\medskip
For a matroid $\M$ on a ground set $[n]$, Fink and Speyer defined a $T$-equivariant $K$-class $y(\M)$ on a Grassmannian $\operatorname{Gr}(r;n)$.
We now define an analogous $T$-equivariant $K$-class $y(\D)$ for a delta-matroid $\D$.  For a feasible set $B$ of $\D$, denote by $\operatorname{cone}_B(\D)$ the tangent cone of $P(\D)$ at the vertex $\be_{B\cap [n]}$, i.e.,
\[
\operatorname{cone}_B(\D) = \RR_{\geq 0}\{P(\D) - \be_{B\cap [n]}\}.
\]
Since $\operatorname{cone}_B(\D)$ is a rational strongly convex cone whose set of primitive rays is a subset of $\mathcal T_B$, the multigraded Hilbert series
\[
\operatorname{Hilb}(\operatorname{cone}_B(\D)\cap \ZZ^n) = \sum_{\mathbf m\in \operatorname{cone}_B(\D)\cap \ZZ^n} T^{-\mathbf m}
\]
is a rational function whose denominator divides $\prod_{\mathbf v\in \mathcal T_B} (1-T^{-\mathbf v})$ \cite[Theorem 4.5.11]{Sta12}.

\begin{prop-defn}\label{prop:yDdef}
For a delta-matroid $\D$ on $[n,\bar n]$, define $y(\D) \in K_T(\operatorname{OGr}(n;2n+1)^T)$ by
\[
y(\D)_B = \begin{cases}
\operatorname{Hilb}(\operatorname{cone}_B(\D) \cap \ZZ^n) \cdot \prod_{\mathbf v\in \mathcal T_B} (1-T^{-\mathbf v}) & \text{if $B$ a feasible set of $\D$}\\
0 &\text{if otherwise}
\end{cases}
\]
for any maximal admissible subset $B \subset [n,\bar n]$.  Then $y(\D)$ lies in the subring $K_T(\operatorname{OGr}(n;2n+1))$. 
\end{prop-defn}

We omit the proof of the proposition, as it is essentially identical to the proof of the analogous statement \cite[Proposition 3.2]{FinkSpeyer} for matroids. Alternatively, it can be deduced from Theorem~\ref{thm:schuberts} and Proposition~\ref{prop:isVal}.
Let us note however the following difference from the matroid case.
For a matroid $\M$ on $[n]$, the class $y(\M)$ in \cite{FinkSpeyer} has the property that if $\L\in \operatorname{Gr}(r;n)$ realizes $\M$, then $y(\M)$ equals $[\mathcal O_{\overline {T\cdot \L}}]$, the $K$-class of the structure sheaf of the torus-orbit closure.
This property often fails for delta-matroids because delta-matroid base polytopes often do not enjoy certain polyhedral properties enjoyed by matroid base polytopes, namely normality and very ampleness.

Recall that a lattice polytope $P \subset \RR^n$ (with respect to the lattice $\ZZ^n$) is \emph{normal} if
for all positive integer $\ell$ one has $(\ell P) \cap \ZZ^n = \{\mathbf m_1 + \dots + \mathbf m_\ell : \mathbf m_i \in P \cap \ZZ^n \text{ for all } i =1, \dots, \ell\}$.  If $P$ is normal, then it is \emph{very ample}, meaning that for every vertex $\mathbf v$ of $P$, one has
\[
\big(\RR_{\geq 0}\{P - \mathbf v\}\big) \cap \ZZ^n = \ZZ_{\geq 0}\{(P - \mathbf v)\cap \ZZ^n\}.
\]

\begin{prop}\label{prop:OTL}
For a delta-matroid $\D$ realized by $\L\in \operatorname{OGr}(n;2n+1)$, the $T$-equivariant $K$-class $[\mathcal O_{\overline{T\cdot \L}}]$ of the structure sheaf of the torus-orbit-closure of $L$ satisfies
\[
[\mathcal O_{\overline{T\cdot \L}}]_B =
\begin{cases}
 \operatorname{Hilb}\big( \ZZ_{\geq 0}\{(P(\D) - \be_{B\cap [n]})\cap \ZZ^n\}\big)\prod_{\mathbf v \in \mathcal T_B} (1- T^{-\mathbf v}) & \text{if $B$ a feasible subset of $\D$}\\
 0 & \text{if otherwise}
 \end{cases}
\]
for any maximal admissible subset $B$.
In particular, the $T$-equivariant $K$-class $y(\D)$ equals $[\mathcal O_{\overline{T\cdot \L}}]$ if and only if $P(\D)$ is very ample.
\end{prop}

\begin{proof}
For a finite subset $\mathscr A \subset \ZZ^n$, let $Y_{\mathscr A}$ be the projective toric variety defined as the closure of the image of the map $T\to \PP^{|\mathscr A|-1}$ given by $\mathbf t \mapsto (\mathbf t^{\mathbf m})_{\mathbf m\in \mathscr A}$. 
Writing $\be_0 = 0 \in \ZZ^n$, let us consider
\[
\mathscr A(L) = \left\{\be_S : \begin{matrix}\text{$S \subset [n,\bar n]\cup \{0\}$ with $|S| = n$ such that}\\ \text{the $S$-th Pl\"ucker coordinate of $L$ is nonzero}\end{matrix}\right\}.
\]
There is an embedding of $\mathbb{P}^{|\mathscr A| - 1}$ into $\PP(\bigwedge^n \kk^{2n+1})$ which identifies  the orbit closure $\overline{T\cdot \L} \subset \PP(\bigwedge^n \kk^{2n+1})$ with $Y_{\mathscr A(L)}$.
We now claim that
\[
\mathscr A(L) = \{\mathbf m + \mathbf m' - (1, \ldots, 1) : \mathbf m, \mathbf m' \in P(\D) \cap \ZZ^n\} \subset \widehat{P(\D)}.\]
That is, up to translation by $-(1, \ldots, 1)$, the set $\mathscr A(L)$ is the set of all sums of two (not necessarily distinct) lattice points in $P(\D)$.
When $B$ is a feasible set of $\D$, in the $T$-invariant affine chart $U_B$ around $L_B$, the coordinate ring $\mathcal O_{\overline{T\cdot \L}}(U_B)$ equals the semigroup algebra $
\kk[\ZZ_{\geq 0}\{ \mathbf m - \be_B : \mathbf m \in \mathscr A(L)\}]$, which the claim implies equals $\kk[\ZZ_{\geq 0}\{(P(\D) - \be_{B\cap [n]})\cap \ZZ^n\}]$, and thus the proposition follows from \cite[Theorem 8.34]{MS05} (see also \cite[Theorem 2.6]{FS10}).

For the claim, we first note that $\mathscr A(L)$ is contained in $\widehat{P(\D)}\cap \ZZ^n$ and contains all vertices of $\widehat{P(\D)}$ because the moment polytope $\mu(\overline{T\cdot \L})$ equals $\widehat{P(\D)}$ by \cite[Proposition 6.2]{EFLS}.
The Pl\"ucker embedding $\operatorname{OGr}(n;2n+1) \hookrightarrow \PP(\bigwedge^n \kk^{2n+1})$ is given by the square $\mathcal O(2)$ of the very ample generator $\mathcal O(1)$ of the Picard group of $\operatorname{OGr}(n;2n+1)$. 
Because homogeneous spaces are projectively normal, we find that $\overline{T\cdot \L}$ is isomorphic to $Y_{\mathscr A}$ for some subset $\mathscr A \subseteq P(\D)\cap \ZZ^n$ that includes all vertices of $P(\D)$.  But all lattices points of $P(\D)$ are its vertices, so $\mathscr A = P(\D) \cap \ZZ^n$.  Therefore, the projective embedding of $\overline{T\cdot \L}$ given by $\mathcal O(2)$ is isomorphic to $Y_{2\mathscr A}$ where $2\mathscr A = \{\mathbf m + \mathbf m' : \mathbf m, \mathbf m' \in \mathscr A\}$, which after translating each element by $-(1,\ldots, 1)$ is exactly $\mathscr A(L)$.
\end{proof}

The polytope $P(\D)$ can fail to be very ample in various degrees.  See \Cref{sec:egs} for a series of examples.  In particular, the class $y(\D)$ may not equal $[\mathcal O_{\overline{T\cdot \L}}]$ when $L$ realizes $\D$.

\begin{rem}
\Cref{prop:OTL} also implies that the class $[\mathcal O_{\overline{T\cdot \L}}]$ depends only on the delta-matroid $\D$, independently of the realization $L$ of $\D$.
The analogous statement fails when delta-matroids are considered as ``type C Coxeter matroids,'' a.k.a.\ symplectic matroids.  More precisely, in \cite{BGW98}, realizations of delta-matroids are points on the Lagrangian Grassmannian $\operatorname{LGr}(n;2n)$ consisting of maximal isotropic subspaces with respect to the standard symplectic form on $\kk^{2n}$.  However, in this case, the $K$-class of the torus-orbit-closure of a point $\L\in \operatorname{LGr}(n;2n)$ may not be determined by the delta-matroid that $L$ realizes. See the following example. This is related to the fact that the parabolic subgroup corresponding to $\operatorname{OGr}(n; 2n+1)$ is \emph{minuscule}, but the parabolic subgroup corresponding to $\operatorname{LGr}(n; 2n)$ is not \cite[Section 2.11]{MR1782635}. 
\end{rem}

\begin{eg}\label{ex:typeC}
Let $\mathbb{C}^4$ (with coordinates labeled by $(1, 2, \bar{1}, \bar{2}$) be equipped with the standard symplectic form.
The torus $T = (\mathbb{C}^*)^2$ acts on $\mathbb{C}^4$ by $(t_1, t_2) \cdot (x_1, x_2, x_{\bar{1}}, x_{\bar{2}}) = (t_1 x_1, t_2x_2, t_1^{-1} x_{\bar{1}}, t_2^{-1}x_{\bar{2}})$.
For each $z \in \mathbb{C}$, consider the $2$-dimensional subspace $L_z$ spanned by $(1, 0, 1, z)$ and $(0, 1, z, 1)$, which is isotropic.
For all $z \not= \pm 1$, every Pl\"{u}cker coordinate corresponding to a maximal admissible subset is nonzero.
Thus, the moment polytope $\mu(\overline{T\cdot [L_z]})$ always equals $[-1,1]^2 \subset \RR^2$ as long as $z \neq \pm 1$.
However, when $z = 0$, one computes that $\overline{T\cdot [L_z]} \simeq \PP^1\times \PP^1$, whereas $\overline{T\cdot [L_z]}$ is a toric surface with four conical singularities when $z \neq \pm1$ and $z\neq 0$.
As a result,  one verifies that the $[\mathcal O_{\overline{T\cdot [L_0]}}] \neq [\mathcal O_{\overline{T\cdot [L_3]}}]$, even as non-equivariant $K$-classes.
\end{eg}

\subsection{$K$-classes on the type B permutohedral variety}\label{ssec:KtypeB}

We explain how the geometry of the type B permutohedral variety $X_{B_n}$ relates to the class $y(\D)$ on $\operatorname{OGr}(n;2n+1)$, which we will use to prove \Cref{thm:FS}.
We begin by briefly reviewing the relation between delta-matroids and $X_{B_n}$, details of which can be found in \cite[Section 2]{EFLS}.

\begin{defn}\label{defn:Bn}
Let $W$ be the \emph{signed permutation group} on $[n,\bar n]$, which is the subgroup of the permutation group $\mathfrak S_{[n,\bar n]}$ defined as
\[
W = \{w\in \mathfrak S_{[n,\bar n]} : w(\overline i) = \overline{w(i)} \text{ for all $i\in [n]$}\}.
\]
The \emph{${B_n}$ permutohedral fan} $\Sigma_{B_n}$ is the complete fan in $\RR^n$, unimodular with respect to the lattice $\ZZ^n$, whose maximal cones are labeled by elements of $W$, with the maximal cone $\sigma_w$ being
\[
\RR_{\geq 0}\{\be_{w(1)}, \be_{w(1)}+\be_{w(2)}, \dots, \be_{w(1)}+ \be_{w(2)}+\dots + \be_{w(n)}\} \quad\text{for each $w\in W$}.
\]
\end{defn}

Let $X_{B_n}$ be the (smooth projective) toric variety of the fan $\Sigma_{B_n}$, which contains $T$ as its open dense torus.  For each $w\in W$, let $\operatorname{pt}_w$ be the $T$-fixed point of $X_{B_n}$ corresponding to the maximal cone $\sigma_w$.
We follow \cite{Ful93, CLS} for toric variety conventions.

\medskip
The normal fan of a delta-matroid polytope $P(\D)$ is always a coarsening of $\Sigma_{B_n}$\cite[Section 4.4]{ACEP20}.  Hence, under the standard correspondence between nef toric line bundles and polytopes, the polytope $P(\D)$ defines a line bundle whose $K$-class we denote $[P(\D)] \in K(X_{B_n})$.  See \cite[Chapter 6]{CLS} and \cite[Section 2.2]{EFLS} for details.
The assignment $\D\mapsto [P(\D)]$ is \emph{valuative} in the following sense.

\begin{defn}
For a subset $S\subset \RR^n$, let $\one_S \colon \RR^n \to \ZZ$ be defined by $\one_S(x) = 1$ if $x\in S$ and $\one_S(x) = 0$ if otherwise.  Define the \emph{valuative group} of delta-matroids on $[n,\bar n]$ to be
\[
\mathbb I(\mathsf{DMat}_n) = \text{the subgroup of $\ZZ^{(\RR^n)}$ generated by $\{\one_{P(\D)} : \D \text{ a delta-matroid on $[n,\bar n]$}\}$}.
\]
A function $f$ on delta-matroids valued in an abelian group is \emph{valuative} if it factors through $\mathbb I(\mathsf{DMat}_n)$.
\end{defn}

We record the following useful consequence of \cite[Theorem D]{EFLS}.

\begin{thm}\label{thm:schuberts}
Let $\mathscr D = \{\D \text{ a delta-matroid on $[n,\bar n]$}: \text{$\D$ has a realization $L$ with $[\mathcal O_{\overline{T\cdot \L}}] = y(\D)$}\}$.
Then, the delta-matroids in $\mathscr D$ generate both the $K$-ring $K(X_{B_n})$, considered as an abelian group, and the valuative group $\mathbb I(\mathsf{DMat}_n)$.  That is, the set $\{[P(\D)]: \D\in \mathscr D\}$ generates $K(X_{B_n})$, and the set $\{1_{P(\D)} : \D \in \mathscr D\}$ generates $\mathbb I(\mathsf{DMat}_n)$.
\end{thm}

\begin{proof}
We first note that the set $\mathscr D$ includes the family of delta-matroids known as \emph{Schubert delta-matroids} \cite[Definition 2.6]{EFLS}.  Indeed, Schubert delta-matroids are realizable \cite[Example 6.3]{EFLS}, and their base polytopes, being isomorphic to a polymatroid polytope, are normal \cite[Chapter 18.6, Theorem 3]{Wel76}.
Hence, by \Cref{prop:OTL}, the set $\mathscr D$ includes all Schubert delta-matroids.
Now, Schubert delta-matroids generate both $K(X_{B_n})$ \cite[Theorem D]{EFLS} and $\mathbb I(\mathsf{DMat}_n)$ \cite[Proposition 2.7]{EFLS}.
\end{proof}

Lastly, the $K$-class $y(\D)$ relates to the geometry of $X_{B_n}$ in the following way.  When $\D$ has a realization $\L\in \operatorname{OGr}(n;2n+1)$, there exists a unique $T$-equivariant map $\varphi_L \colon X_{B_n} \to \operatorname{OGr}(n;2n+1)$ such that the identity point of the torus $T\subset X_{B_n}$ is mapped to $\L$ \cite[Proposition 7.2]{EFLS}.  Note that its image is the torus-orbit-closure $\overline{T\cdot \L}$.

\begin{prop}\label{prop:isVal}
The assignment $\D \mapsto y(\D)$ is the unique valuative map such that $y(\D) = {\varphi_L}_*[\mathcal O_{X_{B_n}}]$ whenever $\D$ has a realization $L$.
\end{prop}

\begin{proof}
The assignment $\D \mapsto y(\D)$ is valuative because taking the Hilbert series of the tangent cone at a chosen point is valuative.  When $\D$ has a realization $L$ and $P(\D)$ is very ample, the map $\varphi_L$, considered as a map $X_{B_n} \to \overline{T\cdot \L}$ of toric varieties, is induced by a map of tori with a connected kernel.  Hence, in this case we have ${\varphi_L}_*[\mathcal O_{X_n}] = [\mathcal O_{\overline{T\cdot \L}}]$ by \cite[Theorem 9.2.5]{CLS} and $ [\mathcal O_{\overline{T\cdot \L}}] = y(\D)$ by \Cref{prop:OTL}.
The uniqueness then follows from \Cref{thm:schuberts}.

To see that $y(\D) = {\varphi_L}_*[\mathcal O_{X_{B_n}}]$ whenever $\D$ has a realization $L$, even if $P(\D)$ is not very ample, we compute the pushforward using the Atiyah--Bott formula.
First, for a maximal admissible $B\subset [n,\bar n]$, the construction of the map $\varphi_L$ shows that the fiber $\varphi_L^{-1}(L_B)$ is
\[
\varphi_L^{-1}(L_B) = \begin{cases}
\left\{\operatorname{pt}_w \in X_{B_n}^T: \begin{matrix} w\in W \text{ such that the dual cone of}\\ \text{$\RR_{\geq 0}\{P(\D) - \be_{B\cap [n]}\}$ contains $\sigma_w$}\end{matrix}\right\} &\text{if $B$ a feasible set of $\D$}\\
\emptyset & \text{if otherwise}.\\
\end{cases}
\]
We note that, because the normal fan of $P(\D)$ is a coarsening of $\Sigma_{B_n}$, for $B$ a feasible set of $\D$, the cones $\{\sigma_w : \operatorname{pt}_w\in \varphi_L^{-1}(L_B)\}$ form a polyhedral subdivision of the dual cone of $\RR_{\geq 0}\{P(\D) - \be_{B\cap [n]}\}$.  Now, the desired result follows from combining \cite[Theorem 5.11.7]{CG10} and the generalized Brion's formula \cite[Theorem 2.3]{Ish90}, \cite{Bri88}.
\end{proof}

\begin{rem}\label{rem:OTL}
One could have defined a $K$-class on $\operatorname{OGr}(n;2n+1)$ for an arbitrary delta-matroid $\D$ via the formula in \Cref{prop:OTL} instead of \Cref{prop:yDdef}.
Abusing notation, denote this alternate $K$-class by $[\mathcal O_{\overline{T\cdot \D}}]$, even though $\D$ may not be realizable.  
\Cref{prop:OTL} states that $y(\D) = [\mathcal O_{\overline{T\cdot \D}}]$ exactly when $P(\D)$ is very ample (with respect to $\ZZ^n$).
Unlike $\D \mapsto y(\D)$, the assignment $\D \mapsto [\mathcal O_{\overline{T\cdot \D}}]$ enjoys the feature that $[\mathcal O_{\overline{T\cdot \D}}] = [\mathcal O_{\overline{T\cdot \L}}]$ whenever $\D$ has a realization $L$, but it is not valuative by \Cref{prop:isVal}.
Moreover, \Cref{thm:FS} fails when $[\mathcal O_{\overline{T\cdot \D}}]$ is used in place of $y(\D)$, and we do not know a description of $\pi_{1*}\pi_n^* \big([\mathcal O_{\overline{T\cdot \D}}]\cdot [\mathcal{O}(1)] \big)$ in terms of known delta-matroid invariants.
See \Cref{sec:egs} for examples and questions about $[\mathcal O_{\overline{T\cdot \D}}]$.
\end{rem}

\section{The exceptional Hirzebruch--Riemann--Roch formula}\label{sec:3}

In this section, we prove Theorem~\ref{thm:HRR}.
We first construct $\psi$ and prove that it is an isomorphism after inverting 2.
Then, we discuss how $\psi$ relates to the \emph{isotropic tautological classes} of delta-matroids constructed in \cite{EFLS}, which we use to finish the proof of \Cref{thm:HRR}.

\subsection{The isomorphism}

We follow the notation and conventions in \cite[Sections 2.1 and 3.1]{EFLS}, recalling what is necessary. For a variety with a $T$-action, we will denote the Chow ring and equivariant Chow ring by $A^\bullet(X)$ and $A_T^\bullet(X)$ respectively. We use the language of moment graphs; see \cite[Section 2.4]{FS10} or \cite[Lecture 2]{MacPherson}.

We first define the moment graph $\Gamma$ associated to the $T$-action on $X_{B_n}$. The vertex set $V(\Gamma)$ is the signed permutation group $W$, which indexes the torus-fixed points of $X_{B_n}$, and the edges $E(\Gamma)$ are given by $(w,w\tau)$ for a transposition $\tau\in \{(1,2),(2,3),\ldots,(n-1,n),(n, \bar{n})\}$, indexing $T$-invariant $\mathbb{P}^1$'s joining torus-fixed points of $X_{B_n}$. Denote $\tau_{i,i+1}:=(i,i+1)$ and $\tau_n:=(n,\bar{n})$. We have edge labels $c(w,w\tau)$ which are characters of $T$ up to sign (i.e., elements of $\mathbb{Z}^n/\pm 1$) by taking $c(w,w\tau_n)=\pm \be_{w(n)}\in \mathbb{Z}^n/\pm 1$ and $c(w,w\tau_{i,i+1})=\pm (\be_{w(i)}-\be_{w(i+1)})\in \mathbb{Z}^n/\pm 1$, recalling the convention that $\be_{\overline{i}}=-\be_i$. For an edge label $c(ij)$, write $c(ij)_k$ for the $k$-component. 

By the identification of the character lattice of $T$ with $\mathbb{Z}^n$, we write $K_T(\operatorname{pt})=\mathbb{Z}[T_1^{\pm 1},\ldots,T_n^{\pm 1}]$ and $A_T^\bullet(\operatorname{pt})=\mathbb{Z}[t_1,\ldots,t_n].$
By equivariant localization we have
$$K_T(X_{B_n})=\{(f_v)_{v\in V(\Gamma)}:f_i-f_j\equiv 0\pmod {1-\prod_{k=1}^nT_k^{c(ij)_k}}\text{ for all }(i, j)\in E(\Gamma)\}\subset \bigoplus_{v\in \Gamma} K_T(\operatorname{pt}),$$
$$A_T^\bullet(X_{B_n})=\{(f_v)_{v\in V(\Gamma)}:f_i-f_j\equiv 0 \pmod{\sum_{k=1}^nc(ij)_k\cdot t_k} \text{ for all }(i, j)\in E(\Gamma)\}\subset \bigoplus_{v\in \Gamma} A^{\bullet}_T(\operatorname{pt}).$$
Note that both compatibility conditions are invariant under $c(ij)\mapsto -c(ij)$. These are algebras over the rings $\mathbb{Z}[T_1^{\pm 1},\ldots,T_n^{\pm 1}]$ and $\mathbb{Z}[t_1,\ldots,t_n]$ respectively, which are identified as subrings of $K_T(X_{B_n})$ and $A^{\bullet}_T(X_{B_n})$ via the constant collections of $(f_v)_{v\in V}$. Additionally, we have that
$$K(X_{B_n})=K_T(X_{B_n})/(T_1-1,\ldots,T_n-1)\text{ and }A^\bullet(X_{B_n})=A^\bullet_T(X_{B_n})/(t_1,\ldots,t_n).$$
Finally, there are $W$-actions on $K_T(X_{B_n})$ by $(w\cdot f)_{w'}(T_1,\ldots,T_n)=f_{w^{-1}w'}(T_{w(1)},\ldots,T_{w(n)})$, and on $A_T(X_{B_n})$ by $(w\cdot f)_{w'}(t_1,\ldots,t_n)=f_{w^{-1}w'}(t_{w(1)},\ldots,t_{w(n)})$, where we set
$$T_{\overline{i}}=T_i^{-1}\text{ and }t_{\overline{i}}=-t_i.$$
This action descends to the usual action of $W \subset \operatorname{Aut} X_{B_n}$ on $K(X_{B_n})$ and $A^{\bullet}(X_{B_n})$.

\begin{thm}
    There is an injective ring map
    $$\psi_T \colon K_T(X_{B_n})\to A^\bullet_T(X_{B_n})[1/(1\pm t_i)]:=A_T^\bullet(X_{B_n})[\{ \textstyle\frac{1}{1-t_i},\frac{1}{1+t_i}\}_{1\le i \le n}]$$ obtained by
    \begin{equation}\label{eq:maponfixed}
           (\psi_T(f))_w(t_1,\ldots,t_n)=f_w\left(\frac{1+t_1}{1-t_1},\ldots, \frac{1+t_n}{1-t_n}\right).
    \end{equation}
    This map descends to a non-equivariant map $\psi \colon K(X_{B_n})\to A^\bullet(X_{B_n})$, which is injective and becomes an isomorphism after tensoring with $\ZZ[\frac{1}{2}]$.
    
    Finally, $\psi_T$ and $\psi$ are $W$-equivariant in the sense that they intertwine the $W$-actions:
    $$\psi_T(w\cdot f)=w\cdot \psi_T(f)\text{ and }\psi(w\cdot f)=w\cdot \psi(f).$$
\end{thm}
    \begin{proof}
    The map $\psi_T$ is defined via the composition 
    $$K_T(X_{B_n}) \to K_T(X_{B_n}^T) \to A^{\bullet}_T(X_{B_n}^T)[\{ \textstyle\frac{1}{1-t_i},\frac{1}{1+t_i}\}_{1\le i \le n}],$$
    where the second map is given by \eqref{eq:maponfixed}.
    We claim the image of this composition lands in the image of the injective map $A^{\bullet}_T(X_{B_n}) \to A^{\bullet}_T(X_{B_n}^T)[\{ \textstyle\frac{1}{1-t_i},\frac{1}{1+t_i}\}_{1\le i \le n}]$. If this is the case, then $\psi_T$ is an injective ring homomorphism, as the maps in the composition are injective ring homomorphisms. 
    We therefore need to check that the compatibility conditions are preserved by $\psi_T$.
    Let $p(z)=\frac{1+z}{1-z}$.
    \begin{itemize}
        \item If $c(ij)=\pm \be_k$, then $f_i(T_1,\ldots,T_n)=f_j(T_1,\ldots,T_n)$ when we set $T_k=1$. Because $p(0)=1$, this implies that $f_i(p(t_1),\ldots,p(t_n))=f_j(p(t_1),\ldots,p(t_n))$ when we set $t_k=0$.
        \item If $c(ij)=\pm (\be_k-\be_\ell)$, then $f_i(T_1,\ldots,T_n)=f_j(T_1,\ldots,T_n)$ when we set $T_k=T_\ell$. This implies that $f_i(p(t_1),\ldots,p(t_n))=f_j(p(t_1),\ldots,p(t_n))$ when we set $t_i=t_j$.
        \item If $c(ij)=\pm (\be_k+\be_\ell)$, then $f_i(T_1,\ldots,T_n)=f_j(T_1,\ldots,T_n)$ when we set $T_k=T_\ell^{-1}$. Because $p(z)=p(-z)^{-1}$, this implies that $f_i(p(t_1),\ldots,p(t_n))=f_j(p(t_1),\ldots,p(t_n))$ when we set $t_k=-t_\ell$.
    \end{itemize}
        We now check that the map $\psi_T$ descends to a map $\psi \colon K(X_{B_n})\to A^\bullet(X_{B_n})$. Note that, under the map $A^\bullet_T(X_{B_n})\to A^\bullet(X_{B_n})$, we have $1\pm t_i\mapsto 1$, so there is an induced map $A^\bullet_T(X_{B_n})[\frac{1}{1\pm t_i}]\to A^\bullet(X_{B_n})$. To obtain the map $\psi$, we have to show that, under the composition $K_T(X_{B_n})\to A^\bullet(X_{B_n})[\frac{1}{1\pm t_i}]\to A^\bullet(X_{B_n})$, the ideal $(T_1-1,\ldots,T_n-1)$ gets mapped to $0$. Indeed, $\psi_T(T_i-1)=\frac{2t_i}{1-t_i}$, which gets mapped to $0$ under the map $A^\bullet_T(X_{B_n})[\frac{1}{1\pm t_i}]\to A^\bullet(X_{B_n})$ because $t_i$ maps to $0$.

        We now check that $\psi$ is an isomorphism after inverting $2$. Note that, under the map $K_T(X_{B_n})\to A^\bullet_T(X_{B_n})[\frac{1}{1\pm t_i}][\frac{1}{2}]$, the element $1+T_i$ maps to the unit $\frac{2}{1-t_i}$, and hence, by the universal property of localization, we have a map $K_T(X_{B_n})[\frac{1}{1+T_i}][\frac{1}{2}]\to A^\bullet_T(X_{B_n})[\frac{1}{1\pm t_i}][\frac{1}{2}]$. 
        We claim that this is an isomorphism. 
        
        Indeed, first note that it is clearly injective by definition of $\psi_T$, so we just have to check surjectivity.
        For $g\in A^\bullet(X_{B_n})[\frac{1}{1\pm t_i}][\frac{1}{2}]$, it is easy to see that $g_w(\frac{T_1-1}{T_1+1},\ldots,\frac{T_n-1}{T_n+1})\in K_T(\operatorname{pt})[\frac{1}{1+T_i}][\frac{1}{2}]$, and arguing as before, we see that
        $$w\mapsto g_w\left(\frac{T_1-1}{T_1+1},\ldots,\frac{T_n-1}{T_n+1}\right)$$
        gives a preimage of $g$ in $K_T(X_{B_n})[\frac{1}{1+T_i}][\frac{1}{2}]$.

        Now the ideal $(T_1-1,\ldots,T_n-1)\subset K_T(X_{B_n})[\frac{1}{1+T_i}][\frac{1}{2}]$ maps to the ideal $(\frac{-2t_1}{1-t_1},\ldots,\frac{-2t_n}{1-t_n})=(t_1,\ldots,t_n)\subset A^\bullet(X_{B_n})[\frac{1}{1\pm t_i}][\frac{1}{2}]$. Hence we obtain that $\psi \otimes \mathbb{Z}[\frac{1}{2}]$ is the isomorphism \begin{align*}K(X_{B_n})\left [\frac{1}{2}\right ]=K_T(X_{B_n})\left [\frac{1}{2}\right ]/(T_1-1,\ldots,T_n-1)&=K_T(X_{B_n})\left [\frac{1}{1+T_i} \right]\left [\frac{1}{2} \right]/(T_1-1,\ldots,T_n-1)\\&\cong A_T^\bullet(X_{B_n})\left [\frac{1}{1\pm t_i} \right] \left [\frac{1}{2} \right]/(t_1,\ldots,t_n)\\&=A_T^\bullet(X_{B_n})\left [\frac{1}{2} \right ]/(t_1,\ldots,t_n)=A^\bullet(X_{B_n})\left [\frac{1}{2} \right].
        \end{align*}
                Finally, we check $W$-equivariance. Let $\epsilon_i(w)$ equal $1$ if $w(i)\in \{1,\ldots,n\}$ and $-1$ if $w(i)\in \{\overline{1},\ldots,\overline{n}\}$. Then, for $f\in K_T(X_{B_n})$, we verify the $W$-equivariance of $\psi_T$ by computing
        \begin{align*}(w\cdot \psi_T(f))_{w'}&=f_{w^{-1}w'}\left(\frac{1+t_{w(1)}}{1-t_{w(1)}},\ldots,\frac{1+t_{w(n)}}{1-t_{w(n)}}\right)\text{, and}\\
        (\psi_T(w\cdot f))_{w'}&=f_{w^{-1}w'}\left(\left(\frac{1+\epsilon_1(w)t_{w(1)}}{1-\epsilon_1(w)t_{w(1)}}\right)^{\epsilon_1(w)},\ldots, \left(\frac{1+\epsilon_n(w)t_{w(n)}}{1-\epsilon_n(w)t_{w(n)}}\right)^{\epsilon_n(w)}\right)\end{align*}
        which are equal as $p(z)=\frac{1+z}{1-z}$ has $p(z)=p(-z)^{-1}$. The $W$-equivariance then descends to $\psi$.
    \end{proof}

\begin{rem}
Although we state the theorem above for $X_{B_n}$, we note that the only hypothesis on the moment graph $\Gamma$ used in the proof up to the verification of $W$-equivariance is that all edge labels lie in the set $\{\pm \be_k:1\le k \le n\} \cup \{\pm (\be_k+\be_\ell):1\le k <\ell\le n\}\cup \{\pm (\be_k-\be_\ell):1\le k<\ell\le n\}$.
\end{rem}

\begin{rem}
The map $\psi \colon K(X_{B_n}) \to A^\bullet(X_{B_n})$ differs from the previous Hirzebruch--Riemann--Roch-type isomorphisms for $X_{B_n}$ established in \cite{EFLS}, but is related as follows.
Let $\phi^B$ and $\zeta^B$ be the exceptional isomorphisms $K(X_{B_n}) \overset\sim \to A^\bullet(X_{B_n})$ as in \cite[Theorem C]{EFLS} and \cite[Proposition 3.7]{EFLS}.
Comparing the formulas for their $T$-equivariant maps, one can show that $\psi$ is the unique ring map such that
\[
\psi([\mathcal L]) = \phi^B([\mathcal L])\cdot \zeta^B([\mathcal L]) \quad\text{for any line bundle $\mathcal L$ on $X_{B_n}$}.
\]
\end{rem}

\subsection{Isotropic tautological classes}

We now discuss the ``isotropic tautological class'' $[\mathcal I_D]\in K(X_{B_n})$ of a delta-matroid $\D$, which was introduced in \cite{EFLS}.  We show how this class is related to $[P(\D)]$ via the $\psi$ map, which will allow us to use the relationship between $[\mathcal I_\D]$ and interlace polynomials established in \cite[Theorem 7.15]{EFLS}.

\medskip
By pulling back the tautological sequence $0\to \mathcal S \to  \mathcal O_{\operatorname{Gr}(n;2n+1)}^{\oplus 2n+1} \to \mathcal Q \to 0$ involving the tautological subbundle and quotient bundle on the Grassmannian, one has a short exact sequence
\begin{equation}\label{eq:SES1}
0\to \mathcal I \to  \mathcal{O}_{\operatorname{OGr}(n;2n+1)}^{\oplus 2n+1} \to \mathcal Q \to 0
\end{equation}
of vector bundles on $\operatorname{OGr}(n;2n+1)$.
For a realization $\L\in \operatorname{OGr}(n;2n+1)$ of a delta-matroid $\D$, pulling back this sequence via $\varphi_L$ yields $T$-equivariant vector bundles $\mathcal I_L$ and $\mathcal Q_L$ on $X_{B_n}$.
In general, we have the following $T$-equivariant $K$-classes for a delta-matroid \cite[Proposition 7.4]{EFLS}.
Denote $T_{\bar{i}} = T_i^{-1}$ for $i\in [n]$, and let $B_w(\D)$ denote the \emph{$w$-minimal feasible set} of $\D$ for $w\in W$, which is the feasible set corresponding to the vertex of $P(\D)$ that minimizes the inner product with any vector $\mathbf v$ in the interior of $\sigma_w$.

\begin{defn}
For a delta-matroid $\D$ on $[n, \bar{n}]$,
define $[\mathcal{I}_{\D}] \in K_T(X_{B_n})$ to be the \emph{isotropic tautological class} of $\D$, given by 
\[
[\mathcal{I}_{\D}]_w = \sum_{i \in B_w(\D)} T_i \quad\text{for all $w\in W$}.
\]
Define $[\mathcal{Q}_{\D}] \in K_T(X_{B_n})$ as $[\mathcal{O}_{X_{B_n}}^{\oplus 2n+1}] - [\mathcal{I}_{\D}]$, that is,
\[
[\mathcal{Q}_{\D}]_w = 1 + \sum_{i\in [n, \bar n]\setminus B_w(\D)} T_i.
\]
\end{defn}

We will use the following fundamental computation relating Chern classes of isotropic tautological classes and interlace polynomials.  For $[\mathcal E]\in K(X_{B_n})$, let $c_i(\mathcal E)$ denote its $i$-th Chern class, and denote by $c(\mathcal E, q) = \sum_{i \geq 0} c_i(\mathcal E)q^i$ its Chern polynomial. Recall that $\gamma$ is the class of the anti-canonical divisor on $X_{B_n}$, which is the line bundle on $X_{B_n}$ corresponding to the cross polytope. 

\begin{thm}\cite[Theorem 7.15]{EFLS}\label{thm:intersection}
Let $\D$ be a delta-matroid on $[n, \bar{n}]$. Then
$$\int_{X_{B_n}} c(\mathcal{I}_{\D}^{\vee}, v) \cdot \frac{1}{1 - \gamma} = (1 + v)^n \operatorname{Int}_{\D}\left(\frac{1 - v}{1 + v}\right).$$
\end{thm}

Many constructions using isotropic tautological classes are valuative  (cf.\ \cite[Proposition 5.6]{BEST}), which is often useful when combined with \Cref{thm:schuberts}.

\begin{lem}\label{lem:Ival}
Any function that maps a delta-matroid $\D$ to a fixed polynomial expression in the exterior powers of $[\mathcal{I}_{\D}]$ or $[\mathcal{Q}_{\D}]$ or their duals is valuative, and similarly for a fixed polynomial expression in the Chern classes of $[\mathcal{I}_{\D}]$ or $[\mathcal{Q}_{\D}]$. 
\end{lem}

\begin{proof}
Let $\mathbb{Z}^{2^{[n, \bar{n}]}}$ be the free abelian group with basis given by subsets of $[n, \bar{n}]$. By \cite[Proposition A.4]{EHL} (see also \cite[Theorem 4.6]{McMullen2009}), the function
$$\{\text{delta-matroids on }[n, \bar{n}]\} \to \bigoplus_{w \in W} \mathbb{Z}^{2^{[n, \bar{n}]}} \text{ given by } \D \mapsto \sum_{w \in W} \be_{B_w(\D)}$$
is valuative. Any such polynomial expression depends only on $B_w(\D)$ for each $w \in W$, and so it factors through this map and is therefore valuative. 
\end{proof}

We also note the following property of Chern classes of $[\mathcal{I}_{\D}]$ and $[\mathcal{Q}_{\D}]$.

\begin{prop}\label{prop:cancel}
Let $\D$ be a delta-matroid. Then $c(\mathcal{I}_{\D}) = c(\mathcal{Q}_{\D}^{\vee})$ and $c(\mathcal{I}_{\D}) c(\mathcal{I}_{\D}^{\vee}) = 1$. 
\end{prop}

\begin{proof}
We claim that one has the following short exact sequence of vector bundles
\[
0 \to \mathcal I \to \mathcal Q^\vee \to \mathcal O_{\operatorname{OGr}(n;2n+1)} \to 0.
\]
The claim implies the proposition for realizable delta-matroids, and by valuativity (Theorem~\ref{thm:schuberts} and Lemma~\ref{lem:Ival}), for all delta-matroids.
For the claim, let $\mathrm{b}$ be the map $\kk^{2n+1} \to (\kk^{2n+1})^\vee$ given by the bilinear pairing of the quadratic form $q$, that is, $\mathrm{b}(x) \colon y \mapsto q(x+y) - q(x) - q(y)$.
Note that if $L \subseteq \kk^{2n+1}$ is isotropic, then $\mathrm{b}(L) \subseteq (\kk^{2n+1}/L)^\vee \subseteq (\kk^{2n+1})^\vee$, since $\mathrm{b}(\ell)(\ell') = q(\ell+\ell') - q(\ell) - q(\ell') = 0$ for all $\ell, \ell'\in L$.
When $\operatorname{char} \kk \neq 2$, the map $\mathrm{b}$ is an isomorphism, and when $\operatorname{char} \kk = 2$, its kernel is $\operatorname{span}(\be_0)$, which is not isotropic.
Hence, the map $\mathrm{b}$ gives
an injection of vector bundles $0\to \mathcal I \to \mathcal Q^\vee$, whose quotient line bundle is necessarily trivial because $\det \mathcal I \simeq \det \mathcal Q^\vee$ from \eqref{eq:SES1}.

Alternatively, one can prove the proposition via localization as follows.  In $K_T(X_{B_n})$, we have that $[\mathcal{I}_{\D}] +1 = [\mathcal{Q}_\D^{\vee}]$, which gives that $c(\mathcal{I}_{\D}) = c(\mathcal{Q}_{\D}^{\vee})$, and therefore that $c(\mathcal{I}_{\D}^{\vee}) = c(\mathcal{Q}_{\D})$. Because $[\mathcal{I}_{\D}] + [\mathcal{Q}_{\D}] = [\mathcal{O}_{X_{B_n}}^{\oplus 2n+1}]$, we have that $c(\mathcal{I}_{\D}) c(\mathcal{Q}_{\D}) = 1$, and substituting gives the result.
\end{proof}

In order to prove Theorem~\ref{thm:HRR}, it remains to prove the Hirzebruch--Riemann--Roch-type formula. We prepare by doing the following computation, which will be used in the proof of Theorem~\ref{thm:FS} as well. Recall that $\widehat{P(\D)} = 2P(\D) - (1, \dotsc, 1)$. 

\begin{prop}\label{prop:compute}
Let $\D$ be a delta-matroid. Then $\psi([P(\D)]) = c(\mathcal{I}_{\D}^{\vee})$. 
\end{prop}

\begin{proof}
The class in $K_T(X_{B_n})$ defined by the line bundle corresponding to $\widehat{P(\D)}$ under the usual correspondence between polytopes and nef toric line bundles on a toric variety has
$$[\widehat{P(\D)}]_w = \prod_{i \in B_w(\D)} T_{\bar{i}}.$$
Therefore, we see that 
$$\psi^T([\widehat{P(\D)}])_w = \prod_{a \in B_w(\D) \cap [n]} \frac{1 -t_a}{1 + t_a} \cdot \prod_{\bar{a} \in B_w(\D) \cap [\bar{n}]} \frac{1 + t_a}{1 - t_a}.$$
On the other hand, by the definition of $[\mathcal{I}_\D]$ and $[\mathcal{Q}_\D]$, we have that
$$c^T(\mathcal{I}_\D)_w = \prod_{i \in B_w(\D)}(1 + t_i), \text{ and } c^T(\mathcal{Q}_\D)_w = \prod_{i \in B_w(\D)}(1 - t_i).$$
We see that $\psi^T([\widehat{P(\D)}]) =  c^T(\mathcal{Q}_{\D})/c^T(\mathcal{I}_\D)$. Because $c(\mathcal{I}_{\D}^{\vee}) = c(\mathcal{I}_{\D})^{-1} = c(\mathcal{Q}_\D)$ by Proposition~\ref{prop:cancel}, we get that
$$\psi([\widehat{P(\D)}]) = \psi([P(\D)]^2) = c(\mathcal{I}_{\D}^{\vee})^2.$$
In a graded ring, a class which has degree zero part equal to $1$ has at most one square root with degree zero part equal to $1$. Using this, we conclude that $\psi([P(\D)]) = c(\mathcal{I}_{\D}^{\vee})$. 
\end{proof}

\begin{proof}[Proof of Theorem~\ref{thm:HRR}]
We have already constructed $\psi$, so it suffices to show that, for any $[\mathcal{E}] \in K(X_{B_n})$, 
$$\chi(X_{B_n}, [\mathcal{E}]) = \frac{1}{2^n} \int_{X_{B_n}} \psi([\mathcal{E}]) \cdot \frac{1}{1 - \gamma}.$$
By Theorem~\ref{thm:schuberts}, $K(X_{B_n})$ is spanned by the classes $[P(\D)]$ for $\D$ a delta-matroid, so it suffices to check this for $[\mathcal{E}] = [P(\D)]$. Note that $\chi(X_{B_n}, [P(\D)])$ is the number lattice points in $P(\D)$, which is the number of feasible sets of $\D$. It follows from Proposition~\ref{thm:intersection} that $\frac{1}{2^n}\int_{X_{B_n}} c(\mathcal{I}_{\D}^{\vee}) \cdot \frac{1}{1 - \gamma}$ is the number of feasible sets of $\D$ as well, so the result follows from Proposition~\ref{prop:compute}.
\end{proof}

\section{The push-pull computation}\label{sec:4}

Our strategy to prove Theorem~\ref{thm:FS} is based on transferring the computation of $\pi_{1*} \pi_n^* (y(\D) \cdot [\mathcal{O}(1)])$ to a computation on $\operatorname{OGr}(n; 2n+1)$. This idea first appeared in \cite[Lemma 4.1]{FinkSpeyer} and was also used in \cite{FlagTutte}. This is implemented in Proposition~\ref{prop:transfer}. We then reduce to a computation on $X_{B_n}$, following the strategy in \cite[Section 10.2]{BEST}.

\begin{prop}\label{prop:transfer}
For $\epsilon \in K(\operatorname{OGr}(n; 2n+1))$, define a polynomial
$$R_{\epsilon}(v) = \sum_{i \ge 0} \chi(\operatorname{OGr}(n; 2n+1), \epsilon \cdot [ \textstyle{\bigwedge}^i \mathcal{Q}^{\vee}]) v^i.$$
Then $\pi_{1*} \pi_n^* \epsilon = R_{\epsilon}(u - 1) \in K(\mathbb{P}^{2n})$, where $u = [\mathcal{O}_H] \in K(\mathbb{P}^{2n})$ is the class of the structure sheaf of a hyperplane $H\subset \PP^{2n}$.
\end{prop}

\begin{proof}
We prove the claim in a slighter more general setting:  Let $X$ be a variety with a short exact sequence of vector bundles $0\to \mathcal S \to \mathcal O_X^{\oplus N} \to \mathcal Q \to 0$.  Let $\PP_X(\mathcal S) = \operatorname{Proj} \operatorname{Sym}^\bullet \mathcal S^\vee$ be the projective bundle with the projection $\pi \colon \PP_X(\mathcal S) \to X$ and the inclusion $\PP_X(\mathcal S) \hookrightarrow X \times \PP^{N-1}$.  Let $\rho \colon \PP_X(\mathcal S) \to \PP^{N-1}$ be the composition $\PP_X(\mathcal S) \hookrightarrow X \times \PP^{N-1} \to \PP^{N-1}$.  We claim that for $\epsilon \in K(X)$, one has
\[
\sum_{i \geq 0} \chi\big(X, \epsilon \cdot [ \textstyle{\bigwedge^i}\mathcal Q^\vee]\big)(u-1)^i = \rho_* \pi^* \epsilon,
\]
where $u$ is the class of the structure sheaf of a hyperplane in $\PP^{N-1}$.

To prove the claim, since $K(\PP^{N-1}) \simeq \ZZ[u]/( u^N)$, and since $\chi(\PP^{N-1}, u^k)$ is equal to $1$ if $0\leq k \leq N-1$ and is equal to $0$ if $k \geq N$, we first note that
\[
\xi = \sum_{i \ge 0} \chi\big(\mathbb{P}^{N-1}, \xi \cdot u^{N -1- i} \cdot (1 - u)\big) u^i  \quad\text{for $\xi \in K(\mathbb{P}^{N-1})$}.
\]
We consider the polynomial
\begin{equation*}\begin{split}
\sum_{i \geq 0} \chi\big( \PP^{N-1}, \rho_* \pi^* \epsilon \cdot u^{N-1-i}(1-u)\big) v^{i} &= \chi\left( \PP^{N-1},  \rho_* \pi^* \epsilon \cdot v^{N} \cdot \frac{1 - u}{v}\cdot \frac{1}{1- u v^{-1}}\right)\\
 &= v^N \chi \left( \PP^{N-1},  \rho_* \pi^* \epsilon \cdot \frac{1}{1+ (1-u)^{-1}(v-1)}\right).
\end{split}\end{equation*}
Letting $\lambda = (1 - u)^{-1} = [\mathcal O(1)] \in K(\PP^{N-1})$ and substituting $v$ with $v+1$, the right-hand-side becomes
\[
(v+1)^N \chi \left( \PP^{N-1},  \rho_* \pi^* \epsilon \cdot \frac{1}{1+\lambda v}\right) = (v+1)^N \chi \left( X,   \epsilon \cdot\pi_* \rho^* \left(\frac{1}{1+\lambda v}\right)\right),
\]
where the equality is due to the projection formula in $K$-theory.
Thus, to finish we need show 
\[
(v+1)^N \pi_* \rho^* \left(\frac{1}{1+\lambda v}\right) = \sum_{i \geq 0} [ \textstyle{\bigwedge^i}\mathcal Q^\vee]v^i.
\]
But this follows by combining the following three facts from \cite[III.8]{Har77} and \cite[A.2]{Eis95}:
\begin{itemize}
\item We have $\pi_*\rho^*(\lambda^i) = [\operatorname{Sym}^i \mathcal S^\vee]$ for all $i\geq 0$.
\item We have $\big(\sum_{i \geq 0}[\bigwedge^i \mathcal S^\vee]v^i\big) \big(\sum_{i\geq 0} [\bigwedge^i \mathcal Q^\vee]v^i \big) = (v+1)^N$ from the dual short exact sequence $0\to \mathcal Q^\vee \to  (\mathcal{O}_X^{\oplus N})^\vee \to \mathcal S^\vee \to 0$.
\item We have $\big(\sum_{i\geq 0} (-1)^i[\operatorname{Sym}^i\mathcal S^\vee]v^i\big)\big(\sum_{i\geq 0}[ \bigwedge^i \mathcal S^\vee]v^i\big) = 1$ from the exactness of the Koszul complex $\bigwedge^\bullet\mathcal S^\vee \otimes \operatorname{Sym}^\bullet \mathcal S^\vee \to \mathcal O_X \to 0$.
\end{itemize}
Lastly, the desired result follows from the general claim by setting $X = \operatorname{OGr}(n;2n+1)$ and $\mathcal S = \mathcal I$, since $\operatorname{OFl}(1,n;2n+1) = \PP_{\operatorname{OGr}(n;2n+1)}(\mathcal I)$.
\end{proof}

Before proving Theorem~\ref{thm:FS}, we make one more preparatory computation. 
\begin{prop}\label{prop:niceChern}
Let $\D$ be a delta-matroid. Then
$$\psi\left (\sum_{p \ge 0} [\wedge^p \mathcal{Q}^{\vee}_{\D}]v^p \right) = (v + 1)^{n+1} \cdot c\left(\mathcal{I}_{\D}, \frac{v-1}{v+1}\right) \cdot c(\mathcal{I}_{\D}).$$
\end{prop}

\begin{proof}
We compute equivariantly. We have that
$$\sum_{p \ge 0} [\wedge^p \mathcal{Q}^{\vee}_{\D}]_w v^p = (1 + v) \prod_{i \in B_w(\D)}(1 + T_i v),$$
see, e.g., \cite[Section 2]{EHL}. Therefore, we get that
\begin{equation*}\begin{split}
\psi^T\left (\sum_{p \ge 0} [\wedge^p \mathcal{Q}^{\vee}_{\D}] \right)_w v^p &= (1 + v) \prod_{i \in B_w(\D)} \left(1 + \frac{1 + t_i}{1 - t_i} v \right) \\ 
&= (1 + v)^{n+1} \prod_{i \in B_w(\D)} \left(1 + \frac{ t_i(v - 1)}{v+1} \right) \cdot \prod_{i \in B_w(\D)} \frac{1}{(1 -t_i)} \\ 
&= (1 + v)^{n+1} \cdot c^T\left(\mathcal{I}_{\D}, \frac{v-1}{v+1} \right) \cdot c^T(\mathcal{I}_{\D}^{\vee})^{-1}.
\end{split}\end{equation*}
As $c(\mathcal{I}_{\D}^{\vee})^{-1} = c(\mathcal{I}_{\D})$ by Proposition~\ref{prop:cancel}, the result follows. 
\end{proof}

\begin{proof}[Proof of Theorem~\ref{thm:FS}]
By Proposition~\ref{prop:transfer}, we need to show that
$$R_{y(\D) \cdot [\mathcal{O}(1)]}(v) := \sum_{p \ge 0} \chi(\operatorname{OGr}(n; 2n+1), y(\D) \cdot [\mathcal{O}(1)] \cdot [\wedge^p \mathcal{Q}^{\vee}]) v^p = (v + 1) \operatorname{Int}_{\D}(v).$$
The left-hand-side is valuative by \Cref{prop:isVal}, and the right-hand-side also by \cite[Theorem 3.6]{ESS}.
Thus, by \Cref{thm:schuberts}, it suffices to verify this equality when $\D$ has a realization $\L\in \operatorname{OGr}(n; 2n+1)$ such that $y(\D) = [\mathcal{O}_{\overline{T \cdot \L}}]$.
As in the proof of \Cref{prop:isVal}, in this case we have a toric map $\varphi_L \colon X_{B_n} \to \overline{T\cdot \L}$ such that ${\varphi_L}_*[\mathcal O_{X_{B_n}}] = y(\D)$, and by construction $\varphi_L^*[\mathcal O(1)] = [P(\D)]$ and $ \varphi_L^*[\wedge^p \mathcal{Q}^{\vee}] = [\wedge^p \mathcal{Q}_{\D}^{\vee}]$.  Hence, by the projection formula, we have that
$$R_{y(\D) \cdot [\mathcal{O}(1)]}(v) = 
 \sum_{p \ge 0} \chi(X_{B_n}, [P(\D)] \cdot [\wedge^p \mathcal{Q}_{\D}^{\vee}]) v^p.$$
Applying Theorem~\ref{thm:HRR} and Proposition~\ref{prop:niceChern}, we get that
\begin{equation*} \begin{split}
R_{y(\D) \cdot [\mathcal{O}(1)]}(v) &= \frac{1}{2^n} \int_{X_{B_n}} \frac{1}{1 - \gamma} \cdot c(\mathcal{I}_{\D}^{\vee}) \cdot (v + 1)^{n+1} \cdot  c\left(\mathcal{I}_{\D}, \frac{v - 1}{v + 1}\right) \cdot c(\mathcal{I}_{\D}) \\ 
& =\frac{(v + 1)^{n+1}}{2^n} \int_{X_{B_n}} \frac{1}{1 - \gamma}   \cdot  c\left(\mathcal{I}_{\D}, \frac{v - 1}{v + 1}\right) \\ 
&= (v + 1) \operatorname{Int}_{\D}(v).
\end{split}\end{equation*}
In the second line we used Proposition~\ref{prop:cancel}, and in the third line we used Proposition~\ref{thm:intersection}.
\end{proof}

\section{Structure sheaves of orbit closures}\label{sec:egs}

We noted in \Cref{rem:OTL} that, using the formula in \Cref{prop:OTL}, one may assign a $K$-class $[\mathcal{O}_{\overline{T \cdot \D}}]$ to a delta-matroid $\D$, different from $y(\D)$.  It has the feature that $[\mathcal{O}_{\overline{T \cdot \D}}] = [\mathcal{O}_{\overline{T \cdot \L}}]$ whenever $\D$ has a realization $\L\in \operatorname{OGr}(n;2n+1)$.
Here, we collect various examples and questions about this $K$-class.
The Macaulay2 code used for the computation of these examples can be found at \url{https://github.com/chrisweur/KThryDeltaMat}.
A database of small delta-matroids can be found at \url{https://eprints.bbk.ac.uk/id/eprint/19837/} \cite{FMN}.

\medskip
We start with the smallest example where $y(\D) \neq [\mathcal{O}_{\overline{T \cdot \D}}]$.

\begin{eg}\label{eg1}
Let $L \subset \kk^7$ be the maximal isotropic subspace given by the row span of the matrix 
$$\begin{pmatrix} 1 & 0 & 0 & 0 & a & b & 0 \\ 0 & 1 & 0 & -a & 0 & c & 0 \\ 0 & 0 & 1 & -b & -c & 0 & 0 \end{pmatrix}$$
for $a, b, c$ generic elements of $\kk$. Then the  delta-matroid $\D$ represented by $L$ has feasible sets 
\[
\{1, 2, 3\}, \{1, \bar{2}, \bar{3}\}, \{\bar{1}, 2, \bar{3}\}, \{\bar{1}, \bar{2}, 3\}.
\]
The stabilizer of $\L$ is $\{(1,1,1), (-1, -1, -1)\} \in T$, so the map $X_{B_3} \to \overline{T \cdot \L}$ is a double cover.
This implies that $y(\D) \not= [\mathcal{O}_{\overline{T \cdot \L}}]$.  Alternatively, one can verify that $P(\D)$ is not very ample with respect to $\ZZ^3$ and use \Cref{prop:OTL}.
We have $\pi_{1*} \pi_n^* ([\mathcal{O}_{\overline{T \cdot \L}}] \cdot [\mathcal{O}(1)])  = R_{[\mathcal{O}_{\overline{T \cdot \L}}] \cdot [\mathcal{O}(1)]}(u-1)$ by Proposition~\ref{prop:transfer}.
A computer computation shows that
$$R_{[\mathcal{O}_{\overline{T \cdot \L}}] \cdot [\mathcal{O}(1)]}(v) = 4v^2 + 8v + 4 = (v + 1)\operatorname{Int}_{\D}(v).$$
In other words, here \Cref{thm:FS} holds with $[\mathcal{O}_{\overline{T \cdot \L}}]$ in place of $y(\D)$ even though $y(\D) \not= [\mathcal{O}_{\overline{T \cdot \L}}]$.
\end{eg}

Let us say that a delta-matroid has property \eqref{star} if \Cref{thm:FS} holds with $[\mathcal{O}_{\overline{T \cdot \D}}]$ in place of $y(\D)$, that is, by \Cref{prop:transfer}, if
\begin{equation}\tag{$*$}\label{star}
R_{[\mathcal{O}_{\overline{T \cdot \D}}] \cdot [\mathcal{O}(1)]}(v) =  (v + 1)\operatorname{Int}_{\D}(v).
\end{equation}
We now feature an example where \eqref{star} fails.

\begin{eg}\label{eg2}
Let $\D$ be the delta-matroid with feasible sets
\begin{equation*}\begin{split}
&\{\bar{1}, \bar{2}, \bar{3}, \bar{4}\}, \{1, \bar{2}, \bar{3}, \bar{4}\}, \{\bar{1}, 2, \bar{3}, \bar{4}\}, \{\bar{1}, \bar{2}, 3, \bar{4}\}, \{\bar{1}, \bar{2}, \bar{3}, 4\},  \{\bar{1}, 2, 3, 4\}, \{1, \bar{2}, 3, 4\}, \{1, 2, \bar{3}, 4\}, \{1, 2, 3, \bar{4}\}.
\end{split}\end{equation*}
A computer computation shows that $(v + 1)\operatorname{Int}_{\D}(v) = 9 + 16v + 7v^2$, but
$$R_{[\mathcal{O}_{\overline{T \cdot \D}}] \cdot [\mathcal{O}(1)]} (v) = 9 + 16v + 6v^2 - v^3 + v^4 + v^5.$$
\end{eg}

A computer search shows that \Cref{eg2} is the only delta-matroid up to $n=4$ that fails \eqref{star}.
The delta-matroids in the above two examples differ in the following ways.  The delta-matroid in \Cref{eg1}
\begin{itemize}
\item is realizable,
\item is \emph{even} in the sense that the parity of $|B\cap [n]|$ is constant over all feasible sets $B$, and
\item has the polytope $P(\D)$ very ample with respect to the lattice (affinely) generated by its vertices.
\end{itemize}
The last property, when $\D$ has a realization $\L$, is equivalent to stating that $\overline{T\cdot \L}$ is a normal variety.
All three properties fail for the delta-matroid in \Cref{eg2}.  We thus ask:

\begin{question}
When does \Cref{thm:FS} hold with $[\mathcal{O}_{\overline{T \cdot \D}}]$ in place of $y(\D)$?  More specifically, is \eqref{star} satisfied when
\begin{itemize}
\item $\D$ is realizable?
\item $\D$ is an even delta-matroid?
\item the polytope $P(\D)$ is very ample with respect to the lattice (affinely) generated by its vertices?
\end{itemize}
\end{question}

We expect \eqref{star} to fail for some realizable delta-matroid, but we do not know any examples.
We conclude with the following realizable even delta-matroid example.

\begin{eg}
Let $G$ be the graph with vertex set $[7]$ and edges $\{12,13,23,34,45,56,57,67\}$. Let $A(G)$ be its adjacency matrix, considered over $\mathbb F_2$ so that it is skew-symmetric with diagonal entries equal to zero.  Let $\D$ be the delta-matroid realized by the row span of the $7\times (7+7+1)$ matrix $[ \  A \ | \ I_7 \ | \ 0 \ ]$.  That is, its feasible sets are
\[
\left\{\begin{matrix} \text{maximal admissible subsets $B \subset [7,\bar 7]$ such that the principal minor}\\ \text{of $A(G)$ corresponding to the subset $B\cap [7]$ is nonzero}\end{matrix}\right\}.
\]
The polytope $P(\D)$ is not very ample with respect to the lattice (affinely) generated by its vertices, demonstrated as follows.
One verifies that $P(\D)$ contains the origin, and the semigroup $\ZZ_{\geq 0}\{P(\D) \cap \ZZ^7\}$ is generated by
\[
\{\be_{12},\be_{13},\be_{23},\be_{34},\be_{45},\be_{56},\be_{57},\be_{67}\}.
\]
In the intersection of the cone $\RR_{\geq 0}\{P(\D)\}$ and the lattice $\ZZ\{P(\D) \cap \ZZ^7\}$, we have the point
\[
(1,1,1,0,1,1,1) = \frac{1}{2}(\be_{12}+\be_{13} + \be_{23}) + \frac{1}{2}(\be_{56}+\be_{57} + \be_{67}) = \be_{13} + \be_{23} - \be_{34} + \be_{45} + \be_{67},
\]
but this point is not in the semigroup $\ZZ_{\geq 0}\{P(\D) \cap \ZZ^7\}$. In particular, the torus-orbit-closure is not normal. 
Nonetheless, this even delta-matroid satisfies \eqref{star}: a computer computation shows that
$$R_{[\mathcal{O}_{\overline{T \cdot \D}}] \cdot [\mathcal{O}(1)]} (v) = 32 + 92v+ 92v^2 + 36v^3+ 4v^4 = (v + 1)\operatorname{Int}_{\D}(v) .$$
\end{eg}

\small
\bibliography{matroid.bib}
\bibliographystyle{alpha}

\end{document}